\documentclass[12pt,a4paper]{article}
\usepackage[latin1]{inputenc}
\usepackage{amsmath, amsthm, amssymb, amsfonts, amscd}
\usepackage[english]{babel}
\usepackage[all]{xy}
\usepackage[linktocpage=true]{hyperref}
\usepackage{color}
\usepackage[table]{xcolor}
\usepackage{graphicx}
\usepackage{pinlabel}
\usepackage{tikz-cd}
 \usepackage{booktabs}
\usepackage{enumitem}
\usepackage{geometry}
\usepackage{mathtools}
\usepackage{authblk}
\usepackage{thmtools}
\usepackage{stmaryrd}

\declaretheoremstyle[bodyfont=\normalfont,spaceabove=\medskipamount,
    spacebelow=\medskipamount]{definition}

\theoremstyle{definition}

\newtheorem{theorem}{Theorem}[section]
\newtheorem{lemma}[theorem]{Lemma}
\newtheorem{corollary}[theorem]{Corollary}
\newtheorem{proposition}[theorem]{Proposition}
\newtheorem{definition}[theorem]{Definition}
\newtheorem{notation}[theorem]{Notation}

\newtheorem{remark}[theorem]{Remark}

\newtheorem{example}[theorem]{Example}

\newtheorem{convention}[theorem]{Convention}

\DeclareMathOperator\tr{tr} 

\title{\large \textbf{ON VASSILIEV INVARIANTS OF VIRTUAL KNOTS}}

\author{\normalsize WOUT MOLTMAKER AND LOUIS H. KAUFFMAN}

\date{}

\begin{document}

\maketitle

\begin{abstract}
We discuss Vassiliev invariants for virtual knots, expanding upon the theory of quantum virtual knot invariants developed in \cite{kauffman2015rotational}. In particular, following the theory of quantum invariants we work with `rotational' virtual knots. We define chord diagrams, weight systems, and give examples of Lie algebra weight systems of rotational virtual knots. We end with a discussion of extended quantum invariants, which capture information that standard quantum invariants of rotational virtuals cannot.
\end{abstract}



\section{Introduction}

Virtual knots and their quantum invariants were first described by the second author in \cite{kauffman1999virtual,Kauffman_introduction,kauffman2015rotational}. Here he also introduced `rotational' virtual knots, which are the diagrammatic analogue of framed knots in the virtual setting. This is in the sense that for rotational virtual knot diagrams, the moves in Figure \ref{fig:kink_deletion} are disallowed (unlike the case for virtual knots). Rotational virtual knots are therefore essentially equivalence classes of virtual knot diagrams under the equivalence generated by the usual moves, except those depicted in Figure \ref{fig:kink_deletion}.

\begin{figure}[ht]
    \centering
    \includegraphics[width=.6\linewidth]{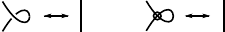}
    \caption{The first Reidemeister move (left) and its virtual analogue (right).}
    \label{fig:kink_deletion}
\end{figure}

Rotational virtual knots were introduced because the moves in Figure \ref{fig:kink_deletion} are precisely those moves that break invariance for the canonical extension of universal quantum invariants to virtual knots. A quantum invariant can be regarded as a partition function with weights that correspond to the crossings and to the maxima and minima in a Morse diagram. The presence of operators for these maxima and minima (the cups and the caps) coupled with virtual crossings means that applying the first virtual Reidemeister move (adding or removing a virtual curl) will change the invariant. This is analogous to the fact that quantum knot invariants are invariants of \textit{framed} knots, rather than of knots, but the change here is more global than the matter of framing. Rotational virtual knots and links are of interest in their own right and in \cite{kauffman2015rotational} a number of combinatorial and quantum invariants of them are articulated. There it is also shown that there are non-trivial rotational virtual links such that the canonical extension of quantum invariants to virtual rotational invariants does not distinguish them from the unlink. Investigating that gap by expanding on the theory of quantum invariants for rotational virtuals is one motivation for the present paper. Another is that quantum invariants associated to quantum groups $U_q(\mathfrak{g})$ are known to give rise to large classes of Vassiliev invariants. As these quantum invariants were defined for rotational virtual knots in \cite{kauffman2015rotational}, it is natural to consider their Vassiliev invariants in a sequel. This motivation also explains why we choose to work exclusively with rotational virtuals in this paper.

In Section \ref{sec:prelim} we briefly discuss the theory of rotational virtual knots, after which we discuss their Vassiliev invariants, chord diagrams, and weight systems at length in Section \ref{sec:vassiliev}. At the end of this section we give a construction of Lie algebra weight systems. Afterwards, in Section \ref{sec:quantum}, we recall the theory of quantum invariants of rotational virtual knots and discuss a generalization of these inspired by virtual chord diagrams. This generalization will turn out to distinguish the rotational virtual links that were shown in \cite{kauffman2015rotational} to be un-distinguishable by any regular quantum invariant.


\section{Rotational Virtual Knots}\label{sec:prelim}

In this section we briefly recall the basic definitions of rotational virtual knots, which were introduced in \cite{kauffman2015rotational}. After these preliminaries we define flat and singular rotational virtual knots.

\subsection{Virtual Knot Theory}

For our purposes it will suffice to think of virtual knots combinatorially, i.e.~as equivalence classes of diagrams rather than as geometric objects. From this perspective, a virtual knot diagram is nothing but a knot diagram with an extra type of crossing:

\begin{definition}
A \textbf{virtual knot diagram} is a $C^\infty$ immersion $S^1\hookrightarrow \mathbb{R}^2$ all of whose singularities are transversal self-intersections that are decorated either with over/under-crossing information (as for classical knot diagrams), or with a circle as in the right-hand side of Figure \ref{fig:kink_deletion}. Self-intersections of the latter kind are referred to as \textbf{virtual crossings}.
\end{definition}

As stated above, virtual knots are equivalence classes of virtual knot diagrams. We consider two equivalence relations on virtual knot diagrams; one yielding virtual knots and the other `rotational' virtual knots:

\begin{definition}
A \textbf{virtual knot} is an equivalence class of virtual knot diagrams, under the equivalence generated by ambient isotopies of $\mathbb{R}^2$, and the moves $R1$, $R2$, $R3$, $vR1$, $vR2$, $vR3$, and mixed $R3$; see Figure \ref{fig:moves}. Note that the `forbidden moves' $F_1$ and $F_2$, depicted in Figure \ref{fig:forbidden} are not included in the set of allowed moves.

\begin{figure}[ht]
    \centering
    \includegraphics[width=.9\linewidth]{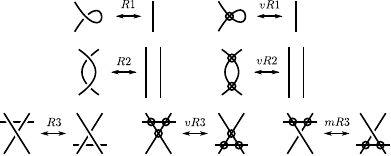}
    \caption{The Reidemeister moves and its virtual analogues.}
    \label{fig:moves}
\end{figure}

\begin{figure}[ht]
    \centering
    \includegraphics[width=.6\linewidth]{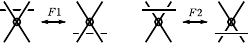}
    \caption{The forbidden moves; generally not allowed for virtual or rotational virtual knots.}
    \label{fig:forbidden}
\end{figure}

A \textbf{rotational virtual knot} is an equivalence class of virtual knot diagrams under the equivalence generated by the same moves as for virtual knots, except for the move $vR1$ and with $R1$ replaced by the \textbf{weakened first Reidemeister move} $R1'$; see Figure \ref{fig:weakened}.

\begin{figure}[ht]
    \centering
    \includegraphics[width=.3\linewidth]{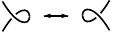}
    \caption{The weakened first Reidemeister move.}
    \label{fig:weakened}
\end{figure}

A virtual knot or rotational virtual knot $S^1\hookrightarrow \mathbb{R}^2$ is called \textbf{oriented} if it is endowed with an orientation inherited from an orientation on $S^1$.
\end{definition}

\begin{remark}
A consequence of $vR2$, $vR3$, and the mixed $R3$ move is the `detour move' on virtual diagrams; see Figure \ref{fig:detour}.

\begin{figure}[ht]
    \centering
    \includegraphics[width=.55\linewidth]{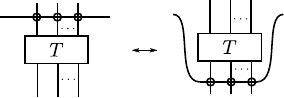}
    \caption{The detour move. Here $T$ is a virtual tangle and the dots indicate any number of parallel strands.}
    \label{fig:detour}
\end{figure}
\end{remark}

Rotational virtual knots are in some sense a virtual version of `framed' knots. We give a brief summary of these below:

\begin{definition}
A \textbf{framed knot} is an equivalence class of knot diagrams in $\mathbb{R}^2$ under the equivalence generated by isotopies of $\mathbb{R}^2$, $R1'$, $R2$, and $R3$. The \textbf{writhe} of a framed knot is its number of positive crossings minus its number of negative crossings.
\end{definition}

The writhe is clearly an invariant of framed knots, whereas it is not invariant under $R1$; see \cite{elhamdadi2020framed} for details. In fact the writhe exactly records the difference between knots and framed knots, in the following sense:

\begin{proposition}
The map
\begin{align*}
    \{\text{Framed knots}\} &\to \{\text{Knots}\} \times \mathbb{Z}\\
    K &\mapsto \left(K,\text{writhe}(K)\right)
\end{align*}
is a bijection.
\end{proposition}

This shows that framed knots are nothing but knots with an integer attached, and justifies that the writhe of a framed knot is often referred to as its `(blackboard) framing'. Note that the framing of a framed knot can be adjusted by the addition of \textbf{curls}, which are the diagrammatic pieces depicted in Figure \ref{fig:weakened}. The rotational nature of rotational virtual knots similarly manifests in the existence of \textbf{virtual curls}: curls with a virtual crossings instead of a classical one. Unlike for virtual knots, in the rotational case these virtual curls cannot simply be removed. Do note that in certain cases they can cancel:

\begin{lemma}
Juxtaposed opposite virtual curls cancel via the `Whitney trick', depicted in Figure \ref{fig:cancellation}.
\end{lemma}

\begin{figure}[ht]
    \centering
    \includegraphics[width=.5\linewidth]{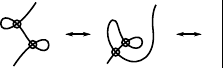}
    \caption{Cancellation of positive curls against negative ones.}
    \label{fig:cancellation}
\end{figure}

Our reason for introducing rotational virtual knots is that we will be interested in \textit{their} Vassiliev invariants, rather than those of virtual knots. This is because the large class of Vassiliev invariants obtained from quantum invariants consists of rotational virtual knot invariants. This is the analogy in the virtual setting to the study of Vassiliev invariants of framed knots.

\begin{convention}
In the rest of this paper, we shall restrict our attention to \textbf{rotational} virtual knots.
\end{convention}

\begin{remark}
The theory of (rotational) virtual knots is related to that of \textit{knotoids}, which were introduced in \cite{turaev2012knotoids}. The relation is as follows: given a knotoid, one can form its \textit{virtual closure} by adding an arc between its end-points that crosses virtually whenever it meets the rest of the knotoid diagram. The result is a virtual knot that is an invariant of the knotoid. Thus invariants of virtual knots yield invariants of knotoids. The virtual closure is well-defined by virtue of the detour move, but it is not well-defined as a rotational virtual knot: one can alter a knotoid diagram by planar isotopy to add virtual curls at will in the virtual closure. See Figure \ref{fig:closure}.

\begin{figure}[ht]
    \centering
    \includegraphics[width=.65\linewidth]{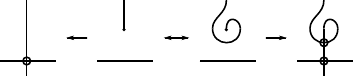}
    \caption{Altering a knotoid to add a virtual curl to its virtual closure.}
    \label{fig:closure}
\end{figure}

This problem is resolved if we work with \textit{biframed} knotoids, which were discussed in \cite{moltmaker2021framed,moltmaker2022new}, and in \cite{gugumcu2021quantum} as \textit{Morse} knotoids. In the case of biframed knotoids changes of coframing correspond to changes in the number of virtual curls under the virtual closure. 
\end{remark}

\subsection{Singular and Flat Rotational Virtuals}\label{subsec:singular}

As noted earlier, virtual knot diagrams can be though of as knot diagrams with an additional decoration choice for self-intersections, other than over/under-crossing information. There are two other decorations that we will consider here, leaving us with a total of 4 possible decorations\footnote{Incidentally we will only encounter diagrams with at most 3 types of crossing, though one could study all 4 types simultaneously, if desired.}.

The first decoration we consider is the undecorated crossing:
\begin{definition}
A \textbf{flat crossing} in a knot diagram is an un-decorated self-intersection of the embedded curve $S^1\to \mathbb{R}^2$ defining the diagram. \textbf{Flat virtual knots} are equivalence classes of knot diagrams with \textit{only} virtual and flat crossings, under the equivalence generated by isotopy of $\mathbb{R}^2$ and the moves depicted in Figure \ref{fig:moves}, but with all classical crossings replaced by flat ones. Similarly \textbf{flat rotational virtual knots} are diagrams with only flat and virtual crossings, up to the equivalence generated by the flat analogue of the moves defining equivalence of rotational virtual knots, \textit{as well as} the flat analogue of $R1$.
\end{definition}

Other than flat and virtual crossings there is another type of crossing that is `neither positive nor negative'. For our purposes it is best understood via the \textit{vector space} of rotational virtual knots:

\begin{definition}
We let $\mathcal{K}$ denote the vector space over $\mathbb{C}$ spanned by a basis consisting of the equivalence classes of rotational virtual knots, and refer to it as the vector space of rotational virtual knots.
\end{definition}

\begin{definition}
A \textbf{singular crossing} is a crossing in a knot diagram decorated with a black dot. A \textbf{singular rotational virtual knot} is a knot diagram with singular crossings, and is understood to be an element of $\mathcal{K}$ via the Vassiliev resolution, depicted in Figure \ref{fig:resolution}: we interpret a singular crossing as the (singular) rotational virtual knot obtained by replacing it with a positive crossing, minus that obtained from replacement with a negative crossing. Repeating this for all singular crossings identifies any singular diagram with an element of $\mathcal{K}$. The \textbf{degree} of a singular rotational virtual knot is its number of singular crossings, and the vector space over $\mathbb{C}$ spanned by singular rotational virtual knots of degree $m$ is denoted $\mathcal{K}_m$.
\end{definition}

\begin{figure}[ht]
    \centering
    \includegraphics[width=.35\linewidth]{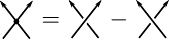}
    \caption{The Vassiliev resolution.}
    \label{fig:resolution}
\end{figure}

An example of a singular rotational virtual knot is depicted in Figure \ref{fig:singular_example}.

\begin{figure}[ht]
    \centering
    \includegraphics[width=.35\linewidth]{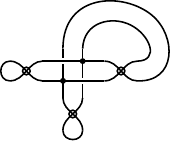}
    \caption{A rotational virtual knot with singular crossings.}
    \label{fig:singular_example}
\end{figure}

The following lemmas are immediate:

\begin{lemma}
There is a filtration
\[
    \mathcal{K} \geq \mathcal{K}_1 \geq \mathcal{K}_2 \geq \dots
\]
meaning $\mathcal{K}_m$ consists consists of singular rotational virtual knots with $m$ or more singular crossings.
\end{lemma}

\begin{lemma}
We have that
\[
    \mathcal{K}_m \cong \frac{\mathcal{K}_m}{\mathcal{K}_{m+1}} \oplus \frac{\mathcal{K}_{m-1}}{\mathcal{K}_m} \oplus \dots \oplus \frac{\mathcal{K}_1}{\mathcal{K}_2} \oplus \frac{\mathcal{K}}{\mathcal{K}_1}
\]
for all $m\in\mathbb{N}_{>0}$.
\end{lemma}

\begin{remark}
When working in $\mathcal{K}_m / \mathcal{K}_{m+1}$, for example with the equivalence class $[K]$ of a degree $m$ singular knot $K$, note that we are allowed to `switch' the classical crossings of $K$ at will. Namely, such a switch amounts to adding or subtracting a degree $m+1$ singular knot to $K$, which does not affect the class $[K]$. So in some sense, when working in $\mathcal{K}_m / \mathcal{K}_{m+1}$ only the information of singular crossings matters, and that of classical crossings can be neglected. We will see this theme recur in Section \ref{sec:vassiliev}.
\end{remark}


\section{Vassiliev Invariants}\label{sec:vassiliev}

In this section we define Vassiliev invariants in terms of the singular rotational virtual knots defined in Section \ref{subsec:singular}. From here we define chord diagrams and weight systems in analogy with the theory of Vassiliev invariants for classical knots. Afterwards we briefly discuss a class of examples of weight systems, namely those coming from representations of semisimple Lie algebras.

\subsection{Vassiliev Invariants}

There is a clear bijection between $\mathbb{C}$-valued rotational virtual knot invariants and linear maps $\mathcal{K}\to \mathbb{C}$. This correspondence is given by sending an invariant $\varphi$ to the linear map taking the values of $\varphi$ on the standard basis elements of $\mathcal{K}$ (which are rotational virtual knots). We define `Vassiliev' invariants to be elements of certain subspaces of $\text{Hom}(\mathcal{K},\mathbb{C})$:

\begin{definition}
A \textbf{degree $m$ Vassiliev invariant} is a $\mathbb{C}$-valued rotational virtual knot invariant $\varphi$ whose associated element of $\text{Hom}(\mathcal{K},\mathbb{C})$ vanishes on $\mathcal{K}_{m+1}$. The vector space of degree $m$ Vassiliev invariants is denoted $\mathcal{V}_m$, and the vector space of all Vassiliev invariants is denoted $\mathcal{V}$.
\end{definition}

The following lemmas are again immediate; see \cite{jackson2019introduction} for details:

\begin{lemma}
We have a filtration
\[
    \mathcal{V}_0 \leq \mathcal{V}_1 \leq \mathcal{V}_2 \leq \dots
\]
as well as isomorphisms of vector spaces
\[
    \mathcal{V}_m \cong \frac{\mathcal{V}_m}{\mathcal{V}_{m-1}} \oplus \frac{\mathcal{V}_{m-1}}{\mathcal{V}_{m-2}} \oplus \dots \oplus \frac{\mathcal{V}_1}{\mathcal{V}_0} \oplus \mathcal{V}_0
\]
for all $m\in\mathbb{N}_{>0}$.
\end{lemma}

\begin{lemma}
We have
\[
    \frac{\mathcal{V}_m}{\mathcal{V}_{m-1}} \cong \left( \frac{\mathcal{K}_m}{\mathcal{K}_{m+1}} \right)^*
\]
for all $m\in\mathbb{N}_{>0}$.
\end{lemma}

\begin{example}\label{ex:jones}
As an example of Vassiliev invariants of (rotational) virtual knots, we consider the Jones polynomial $J$ of virtual knots. Note that this is also an invariant of rotational virtual knots, by pre-composing with the canonical map sending rotational virtual knots to virtual knots. The Jones polynomial is known to satisfy a skein relation, depicted in Figure \ref{fig:jones_skein}.

\begin{figure}[ht]
    \centering
    \includegraphics[width=.65\linewidth]{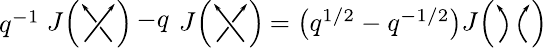}
    \caption{The skein relation satisfied by the Jones polynomial.}
    \label{fig:jones_skein}
\end{figure}

Using this skein relation we can extract Vassiliev invariants from the Jones polynomial as follows: if we substitute $q=e^{h/2}$ into the Jones polynomial and expand the result as a power series in $h$ via $e^{h/2}=1+\frac{h}{2}+\frac{h^2}{8}+\frac{h^3}{48}+\dots$ then one can check that the value attributed to a singular crossing is divisible by $h$, namely by reading off the terms lowest order in $h$ from Figure \ref{fig:jones_skein}. We deduce that the Jones polynomial of a knot with $m$ singular crossings is divisible by $h^m$. As a result, we can define $c_i:\mathcal{K}\to \mathbb{C}$ to be the degree $i$ coefficient of $J(K)\vert_{q=e^{h/2}}$ after expansion as a power series, and find that this is a Vassiliev invariants of order $i$.
\end{example}

\begin{example}\label{ex:quantum}
The previous example applies equally well to the case of more general quantum invariants, where one is given a quantum group $U_q(\mathfrak{g})$ and a finite-dimensional irreducible representation $V$. This quantum group $U_q(\mathfrak{g})$ has the structure of a ribbon Hopf algebra, which we will come back to in Section \ref{sec:quantum}. Here it suffices to say that the invariants associated to $\left(U_q(\mathfrak{g}),V\right)$ assign operators $R,R^{-1}\in \text{Hom}(V\otimes V)$ to positive and negative crossings respectively. It is known that these operators satisfy $R\equiv R^{-1}\text{ mod }h$ \cite{chmutov2012introduction} after the substitution $q=e^{h/2}=1+\frac{h}{2}+\frac{h^2}{8}+\dots$. Therefore the same reasoning from Example \ref{ex:jones} applies to all these examples, retrieving a power series of Vassiliev invariants from each choice of $(\mathfrak{g},V)$.
\end{example}

\subsection{Chord Diagrams}

As in the case of classical knots, Vassiliev invariants are susceptible to study by combinatorial methods. This is facilitated by the introduction of \textit{chord diagrams}. 

We begin with a brief recap of chord diagrams for (framed) knots to guide our intuition for the virtual case. One way to think about the association of chord diagrams to classical (framed) singular knots is as follows: First, at each singular crossing, two portions of the knot cross. Place a dotted `chord' of line segment between these portions of knot (at some points in open neighbourhoods of the singular crossing between these arcs). Next, replace all classical and singular crossings by flat crossings. The result is a flat knot diagram with several points on it connected by chords. Since all flat knots (without virtual crossings) are trivial, we can unknot this diagram. We do so, and keep track of the positions of all the chords along the way. Here crossings between chord and knot diagram are also taken to be flat. The result is a circle decorated with chords; a chord diagram. For an example of this construction, see Figure \ref{fig:knot_chord}.

\begin{figure}[ht]
    \centering
    \includegraphics[width=.6\linewidth]{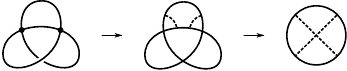}
    \caption{Extracting a chord diagram from a singular knot.}
    \label{fig:knot_chord}
\end{figure}

We consider the vector space $\mathcal{C}$ of chord diagrams of knots taken up to planar isotopy and quotient by the \textbf{4-term relation} $4T$, depicted in Figure \ref{fig:4T}, to obtain the vector space $\mathcal{A}$ of chord diagrams. Note that the chord diagrams in Figure \ref{fig:4T} may have more chords in the portions of chord diagram that aren't depicted, but that we assume these to be identical for all 4 terms.

\begin{figure}[ht]
    \centering
    \includegraphics[width=.75\linewidth]{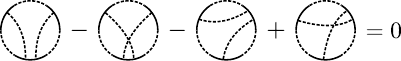}
    \caption{The 4-term relation $4T$.}
    \label{fig:4T}
\end{figure}

For rotational virtual knots, we mimic this interpretation of chord diagrams: given a singular rotational virtual knot, we place chords connecting the arcs of its singular crossings as for knots, and subsequently delete the decorations on classical and singular crossings. The result is a flat rotational virtual knot, decorated with chords. (Recall that we defined flat virtual knots in Section \ref{subsec:singular}.) For an example, see Figure \ref{fig:RV_chord_ex}.

\begin{figure}[ht]
    \centering
    \includegraphics[width=.75\linewidth]{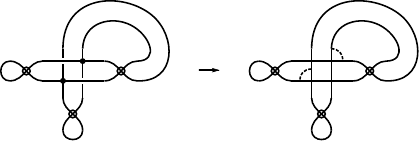}
    \caption{An example of a rotational virtual chord diagram assigned to a singular rotational virtual knot diagram via $\phi_2$.}
    \label{fig:RV_chord_ex}
\end{figure}

\begin{definition}
We define \textbf{chord diagrams of rotational virtual knots} to be flat rotational virtual knots decorated with chords. We call the underlying flat rotational virtual knot the \textbf{skeleton} of the chord diagram. We consider these `rotational virtual chord diagrams' up to the equivalence generated by the same moves that generate the equivalence of flat rotational virtual knots, keeping track of any chord attachment points under such moves.

Note that when tracking chord attachment points under ambient isotopy, we allow these attachment points to slide past flat crossings, but not past virtual crossings. See figure \ref{fig:chord_sliding}.

\begin{figure}[ht]
    \centering
    \includegraphics[width=.7\linewidth]{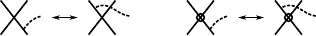}
    \caption{An allowed chord slide (left) and a forbidden virtual chord slide (right).}
    \label{fig:chord_sliding}
\end{figure}

\begin{example}\label{ex:virtual_slide}
To justify that virtual chord slides should be disallowed, consider the rotational virtual chord diagrams depicted at the top of Figure \ref{fig:chord_slide_ex}.

\begin{figure}[ht]
    \centering
    \includegraphics[width=.9\linewidth]{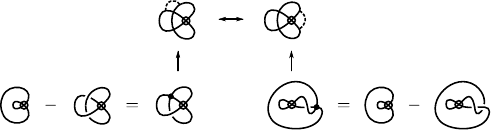}
    \caption{Chord diagrams related by a virtual chord slide but arising from different singular rotational virtual knots.}
    \label{fig:chord_slide_ex}
\end{figure}

It is immediate that the two chord diagrams at the top of Figure \ref{fig:chord_slide_ex} are related by a virtual chord slide. On the row below these chord diagrams are two singular rotational virtual knots that give rise to these chord diagrams. As depicted, one can compute that these singular diagrams correspond to different elements of $\mathcal{K}$. Moreover by virtue of the low crossing number it is easy to see that they are different elements of $\mathcal{K}/\mathcal{K}_2$. Indeed, one can compute they are distinguished by the polynomial invariant $\mathbf{p}_t(K)$ from \cite{Allison}, which is a degree one Vassiliev invariant of non-rotational virtual knots. Therefore we conclude that allowing virtual chord slides contradicts the property that rotational virtual chord diagrams are uniquely determined by their underlying singular diagram, up to equivalence and swapping of classical crossings. Since this property is of fundamental importance to the algebraic formalism of Vassiliev invariants (see Section \ref{subsec:weights}) we are forced to disallow virtual chord slides.
\end{example}

We use $\mathcal{C}^{RV}$ to denote the vector space of chord diagrams of rotational virtual knots.
\end{definition}

\begin{notation}
In this paper we will reserve the notation $\mathcal{C}$ and $\mathcal{A}$ for chord diagram vector spaces of classical knots, as we will make use of such diagrams again later on in the paper. For chord diagrams of rotational virtual knots we will use $\mathcal{C}^{RV}$ and $\mathcal{A}^{RV}$, even though we've chosen not to add such superscripts onto $\mathcal{K}$ and $\mathcal{V}$.
\end{notation}

\begin{definition}\label{def:forget}
Given a rotational virtual chord diagram, we can `forget virtual structure' by replacing all its virtual crossings by flat crossings. Up to equivalence of flat virtual knots, this flat knot diagram is trivial, and hence defines a chord diagram of knots. Extending this construction linearly yields a `forgetful' map of vector spaces $F_v:\mathcal{C}^{RV}\to \mathcal{C}$.
\end{definition}

Having defined rotational virtual chord diagrams, we return to the assignment of rotational virtual chord diagrams to rotational virtual knots described above:

\begin{definition}
Let $\mathcal{K}^\bullet_m$ denote the set of equivalence classes of degree $m$ singular rotational virtual knots. For each $m\in\mathbb{N}$ we define the set function $\phi_m: \mathcal{K}^\bullet_m\to \mathcal{C}^{RV}_m$ via the construction depicted in Figure \ref{fig:RV_chord_ex}: placing chords connecting the two arcs at each virtual crossing and subsequently turning all classical and singular crossings flat.
\end{definition}

\begin{remark}
If we are given an oriented rotational virtual knot $K$, then the above construction goes through as described, except that the chord diagram $\phi_m(K)$ will now have an oriented skeleton. We call such chord diagrams \textbf{oriented}. This orientation on the skeleton is inherited in the obvious way from $K$. Most of this paper goes through equivalently for both oriented and un-oriented chord diagrams, so we will not reserve distinguished notation for the oriented chord diagrams. If the distinction between oriented chord diagrams and un-oriented ones is important, we will make explicit with which we work.
\end{remark}

It is known in the case of classical knots that the chord diagram $\phi_m(K)$ of a singular framed knot $K$, considered as an element of $\mathcal{A}$ instead of $\mathcal{C}$, determines its flattened singular diagram (i.e.~its equivalence class in $\mathcal{K}_m/\mathcal{K}_{m+1}$) uniquely. Inspired to mimic this result for the virtual case, we consider a quotient of rotational virtual chord diagrams analogous to $\mathcal{A}$.

\begin{definition}
We define the 4-term relation $4T$ in the same way as for knots, except that now chords no longer lie between portions of unknotted circle, but between arcs of flat rotational virtual knot diagram. As such $4T$ is more faithfully depicted for rotational virtual chord diagrams as in Figure \ref{fig:4T_virtual}.

\begin{figure}[ht]
    \centering
    \includegraphics[width=.75\linewidth]{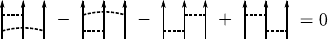}
    \caption{The $4T$ relation for rotational virtual chord diagrams.}
    \label{fig:4T_virtual}
\end{figure}

Other than the $4T$ relation, there is another relation we will want to quotient $\mathcal{C}^{RV}$ by: the \textbf{chord detour} relation $CD$, depicted in Figure \ref{fig:CD}. For some intuition: as arcs with only singular crossings have been allowed to detour in any kind of diagram so far, it is natural that they can also detour around chords.

\begin{figure}[ht]
    \centering
    \includegraphics[width=.4\linewidth]{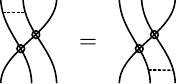}
    \caption{The chord detour relation $CD$.}
    \label{fig:CD}
\end{figure}
\end{definition}

\begin{definition}
We define the vector space $\mathcal{A}^{RV}$ of rotational virtual chord diagrams to be the quotient of $\mathcal{C}^{RV}$ by the subspace generated by the $4T$ and $CD$ relations.
\end{definition}

\begin{lemma}\label{lm:double_slide}
As noted in Example \ref{ex:virtual_slide} virtual chord slides cannot be applied in general. However by virtue of the $CD$ relation \textit{double} chord slides can be applied in $\mathcal{A}^{RV}$, see Figure \ref{fig:double_slide}.

\begin{figure}[ht]
    \centering
    \includegraphics[width=.25\linewidth]{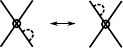}
    \caption{A double chord slide.}
    \label{fig:double_slide}
\end{figure}
\end{lemma}

\begin{proof}
The proof is depicted in Figure \ref{fig:double_slide_pf}.
\begin{figure}[ht]
    \centering
    \includegraphics[width=.8\linewidth]{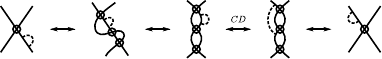}
    \caption{Proof of Lemma \ref{lm:double_slide}.}
    \label{fig:double_slide_pf}
\end{figure}
\end{proof}


To end this section we consider one more construction that will be of interest to us: the `chord-contracting' map $\psi_m: \mathcal{A}^{RV}_m\to \mathcal{K}_m/\mathcal{K}_{m+1}$ that produces from a chord diagram $C$ an equivalence class $[K]\in \mathcal{K}_m/\mathcal{K}_{m+1}$ such that $\phi_m(K)=C$.

\begin{definition}
Given a rotational virtual chord diagram $C$, we define $\psi_m(C)$ to be the element of $\mathcal{K}_m/\mathcal{K}_{m+1}$ represented by the diagram $K$ obtained from replacing all the flat crossings in $C$ by arbitrarily chosen classical crossings, and `contracting' all the chords of $C$ as in Figure \ref{fig:contraction}. Note in this figure that the attaching points of the chord may be separated by chords or portions of virtual knot diagram that are not depicted.

\begin{figure}[ht]
    \centering
    \includegraphics[width=.3\linewidth]{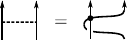}
    \caption{Contracting a chord.}
    \label{fig:contraction}
\end{figure}

Next note that when we contract a chord we create a classical crossing. We may also cross other parts of the diagram, not shown in Figure \ref{fig:contraction}, when contracting a chord, creating even more classical crossings. We choose whether these crossings are positive or negative arbitrarily. As we only consider $[K]\in \mathcal{K}_m/\mathcal{K}_{m+1}$, these choices do not influence the value of $\psi_m(C)$.
\end{definition}

Since the domain of $\psi_m$ is taken to be $\mathcal{A}^{RV}_m$ we must still verify that this construction satisfies the $4T$ and $CD$ relations:

\begin{lemma}
The maps $\psi_m$ are well-defined.
\end{lemma}
\begin{proof}
The fact that $\psi_m$ satisfies $4T$ is known from the case of classical knots \cite{jackson2019introduction}. That $\psi_m$ satisfies $CD$ follows immediately from applying the detour move to the region of diagram shown on the right in Figure \ref{fig:contraction}.
\end{proof}

\subsection{Weight Systems}\label{subsec:weights}

In this section we set up some basic algebraic formalism for Vassiliev invariants mimicking that for the classical case. For this purpose, this subsection follows the notation of \cite{jackson2019introduction}. We begin with the following observation:

\begin{lemma}
Let $\theta\in\mathcal{V}_m$ and $K,K'\in\mathcal{K}^\bullet_m$. Assume that $\phi_m(K)=\phi_m(K')$. Then $\theta(K)=\theta(K')$.
\end{lemma}
\begin{proof}
By construction of the definition of rotational virtual chord diagrams it follows that if $\phi_m(K)=\phi_m(K')$ then $K$ and $K'$ are related by crossing changes and equivalences of rotational virtual knots. Therefore as $K$ and $K'$ are singular of degree $m$ and differ only by elements of $\mathcal{K}_{m+1}$ we must have $\theta(K)=\theta(K')$.
\end{proof}

With this lemma one easily concludes the following:

\begin{corollary}
Let $\theta\in\mathcal{V}_m$. Then $\theta$ factors uniquely through $\phi_m$ via a map $\alpha_m(\theta):\mathcal{C}^{RV}_m\to \mathbb{C}$ to give a commutative diagram
\[
\begin{tikzcd}
\mathcal{K}_m^\bullet \arrow[rd, "\theta"'] \arrow[r, "\phi_m"] & \mathcal{C}^{RV}_m \arrow[d, "\alpha_m(\theta)", dashed] \\
                                                        & \mathbb{C}                                              
\end{tikzcd}
\]
\end{corollary}
\begin{proof}
Clearly to satisfy the commutativity $\alpha_m(\theta)$ must be defined by $C\mapsto \theta(K)$, where $K$ satisfies $\phi_m(K)=C$. Such $K$ exists as $\phi_m$ is clearly surjective. So the result follows if this assignment is well-defined, i.e.~independent of the choice of $K$. Suppose $L$ is any singular rotational virtual knot such that $\phi_m(L)=C$. Then $K$ and $L$ must differ by crossing changes of non-singular crossings. Since $K$, $L$, and $\theta$ are all degree $m$, it follows that $\theta(K)=\theta(L)$ as required.
\end{proof}

As a result, we obtain a linear map
\[
    \alpha_m: \mathcal{V}_m \to (\mathcal{C}^{RV}_m)^*
\]
that will be of particular interest to us.


The kernel of $\alpha_m$ follows immediately from the definitions (see the proof of \cite[Lm.~11.18]{jackson2019introduction} for details):

\begin{lemma}
We have that $\text{ker}(\alpha_m) = \mathcal{V}_{m-1}$ for all $m\in\mathbb{N}_{\geq1}$.
\end{lemma}

We conclude that each $\alpha_m$ factors through an injective map
\[
    \overline{\alpha}_m: \frac{\mathcal{V}_m}{\mathcal{V}_{m-1}} \to (\mathcal{C}^{RV}_m)^*.
\]

Next we are interested in the image of this map. In particular we have the following lemma:
\begin{lemma}\label{lm:weight_sys}
Let $\theta\in \mathcal{V}_m$. Then $\alpha_m(\theta)$ respects both the $4T$ and $CD$ relations, i.e.~vanishes on the subspace we quotient $\mathcal{C}^{RV}_m$ by to obtain $\mathcal{A}^{RV}$.
\end{lemma}
\begin{proof}
Let $W= \alpha_m(\theta)$. The statement for the $4T$ relation is \cite[Lm.~11.24]{jackson2019introduction}. For the $CD$ relation note that for any degree $m$ singular rotational virtual knot $K$ we have $W(\phi_m(K))=\theta(K)$ by definition. From this we conclude the identities given in Figure \ref{fig:CD_proof}.
\begin{figure}[ht]
    \centering
    \includegraphics[width=\linewidth]{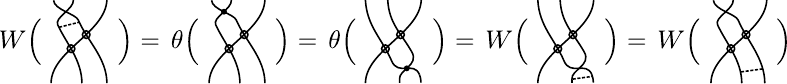}
    \caption{Proof that $\alpha_m(\theta)$ satisfies the $CD$ relation.}
    \label{fig:CD_proof}
\end{figure}

Figure \ref{fig:CD_proof} shows the result in the pretense of a flat crossing near the $CD$ relation. Since flat crossings can be created (and later removed) in pairs anywhere at will by an application of the flat $R2$ relation, the desired result follows.
\end{proof}

Lemma \ref{lm:weight_sys} tells us that each $\overline{\alpha}_m$ produces linear maps $\mathcal{C}^{RV}_m\to \mathbb{C}$ that factor through $\mathcal{A}^{RV}_m$. We are naturally interested in such maps, as they define elements of the dual of $\mathcal{A}^{RV}$. We call these maps \textit{weight systems}:

\begin{definition}
We define a degree $m$ \textbf{weight system} to be a linear map $W:\mathcal{A}^{RV}_m\to \mathbb{C}$. Note that weight systems are in bijection with linear maps $\mathcal{C}^{RV}_m\to \mathbb{C}$ factoring through the quotient of $\mathcal{C}^{RV}_m$ by the $4T$ and $CD$ relations. We denote the vector space of degree $m$ weight systems by $\mathcal{W}_m$, and the vector space of all weight systems by $\mathcal{W}$. In other words, $\mathcal{W}_m=(\mathcal{A}^{RV}_m)^*$ and $\mathcal{W}=(\mathcal{A}^{RV})^*$.
\end{definition}

So we conclude $\text{Im}(\overline{\alpha}_m)\subseteq \mathcal{W}_m$. It is now natural to ask whether this inclusion is in fact an equality. Although we have not developed the tools here to answer this question, we give a brief discussion below of what would be needed to do so, and what the implications of an affirmative answer would be. The following algebraic lemma, which is again an immediate consequence of the definitions, will be helpful:

\begin{lemma}\label{lm:duals}
Under the identifications
\[
    \frac{\mathcal{V}_m}{\mathcal{V}_{m-1}} \cong \left( \frac{\mathcal{K}_m}{\mathcal{K}_{m+1}} \right)^*
    \qquad \text{ and }\qquad
    (\mathcal{A}^{RV}_m)^* = \mathcal{W}_m,
\]
we have $\psi_m^*=\overline{\alpha}_m$.
\end{lemma}

By construction, $\overline{\alpha}_m$ is injective. As a corollary of this lemma, we deduce $\psi_m$ is surjective. We would like to show that $\overline{\alpha}_m$ is moreover surjective, for this would provide an isomorphism $\mathcal{V}_m/\mathcal{V}_{m-1}\cong \mathcal{W}_m$. This isomorphism would tell us that studying Vassiliev invariants is equivalent to studying $(\mathcal{A}^{RV})^*$, vindicating our definition of $\mathcal{A}^{RV}$.

By Lemma \ref{lm:duals}, to show that $\overline{\alpha}_m$ is surjective it suffices to show $\psi_m$ is injective. This would follow easily if one possessed a `universal Vassiliev invariant', which takes values in the vector space $\hat{\mathcal{A}}^{RV}$, the graded completion of $\mathcal{A}^{RV}$ consisting of infinite formal linear combinations of chord diagrams.

\begin{definition}\label{def:universal}
A \textbf{universal Vassiliev invariant} of rotational virtual knots is a rotational virtual knot invariant $\check{Z}:\mathcal{K}\to \hat{\mathcal{A}}^{RV}$ such that for any singular knot $K$, the lowest order term of $\check{Z}(K)$ is equal to the chord diagram of $K$. Note that, by surjectivity of $\psi_m$, this is equivalent to demanding that
\[
    \check{Z}(\psi_m(C)) = C + [\text{terms of degree} > m]
\]
for all $C\in \mathcal{A}^{RV}_m$.
\end{definition}

In the classical theory for framed knots such a universal Vassiliev invariant is provided by the combinatorial Kontsevich invariant \cite{ohtsuki2002quantum}. The extension of this invariant to the rotational virtual setting therefore makes for a worthwhile direction for further research. 
We end this subsection with some consequences to the existence of a universal Vassiliev invariant. So for the remainder of this subsection, we work under the assumption that there exists a universal Vassiliev invariant $\check{Z}$ of rotational virtual knots.


\begin{proposition}\label{prop:universal}
Assume we have a universal Vassiliev invariant $\check{Z}$ of rotational virtual knots. Then there exists a map $\check{Z}_m:\mathcal{K}_m/\mathcal{K}_{m+1}\to \mathcal{A}_m^{RV}$ acting element-wise as
\[
    \check{Z}_m: [K]\mapsto p_m(\check{Z}(K)),
\]
where $p_m$ is the projection $\hat{\mathcal{A}}\to\mathcal{A}_m$. Moreover $\psi_m$ is a vector space isomorphism with $\psi_m^{-1} = \check{Z}_m$.
\end{proposition}
\begin{proof}
By definition of $\check{Z}$ we have $\check{Z}(\mathcal{K}_m)\subseteq \hat{\mathcal{A}}^{RV}_{\geq m}$ and hence we obtain a map
\[
    \check{Z}_{\geq m} : \mathcal{K}_m\to \hat{\mathcal{A}}^{RV}_{\geq m}.
\]
Composing this map with the projection $p_m: \hat{\mathcal{A}}^{RV}_{\geq m}\to \hat{\mathcal{A}}^{RV}_{\geq m}/\hat{\mathcal{A}}^{RV}_{\geq m+1}\cong \mathcal{A}^{RV}_{m}$ then gives a map
\[
    p_m \circ \check{Z}_{\geq m} : \mathcal{K}_m\to \mathcal{A}^{RV}_{m}.
\]
Again by the definition of $\check{Z}$, this map vanishes on $\mathcal{K}_{m+1}\subseteq \mathcal{K}_m$, and hence factors into a map
\[
    \check{Z}_m: \frac{\mathcal{K}_m}{\mathcal{K}_{m+1}} \to \mathcal{A}^{RV}_m
\]
that is as required for the theorem. The fact that $\check{Z}_m\circ \psi_m = \text{id}$ is immediate from the definition of $\check{Z}$. Since $\psi_m$ has a left inverse it must be injective and therefore an isomorphism. We conclude that $\psi_m^{-1} = \check{Z}_m$.
\end{proof}


\begin{corollary}\label{cor:equivalence}
Assuming we have a universal Vassiliev invariant $\check{Z}$, the map $\overline{\alpha}_m$ provides an isomorphism $\mathcal{V}_m/\mathcal{V}_{m-1}\cong \mathcal{W}_m$ for all $m$. We therefore obtain an isomorphism
\[
\begin{tikzcd}
\mathcal{V}_m \arrow[r, "\cong", phantom] \arrow[d, "\cong"'] & \frac{\mathcal{V}_m}{\mathcal{V}_{m-1}} \arrow[r, "\oplus", phantom] \arrow[d, "\overline{\alpha}_m"'] & \frac{\mathcal{V}_{m-1}}{\mathcal{V}_{m-2}} \arrow[r, "\oplus", phantom] \arrow[d, "\overline{\alpha}_{m-1}"'] & \dots \arrow[r, "\oplus", phantom] \arrow[d, "\dots", phantom] & \frac{\mathcal{V}_1}{\mathcal{V}_{0}} \arrow[d, "\overline{\alpha}_1"'] \arrow[r, "\oplus", phantom] & \mathcal{V}_0 \arrow[d, "\overline{\alpha}_0"'] \\
\mathcal{W}_{\leq m} \arrow[r, "\cong", phantom]              & \mathcal{W}_m \arrow[r, "\oplus", phantom]                                                             & \mathcal{W}_{m-1} \arrow[r, "\oplus", phantom]                                                                 & \dots \arrow[r, "\oplus", phantom]                             & \mathcal{W}_1 \arrow[r, "\oplus", phantom]                                                           & \mathcal{W}_0                                  
\end{tikzcd}
\]
telling us that the study of Vassiliev invariants is equivalent to that of weight systems, which is in turn equivalent to studying the structure of $\mathcal{A}^{RV}$.
\end{corollary}


Given a weight system $W:\mathcal{A}^{RV}\to\mathbb{C}$, recalling from the proof of Proposition \ref{prop:universal} that $\check{Z}(\mathcal{K}_{m+1})\subseteq \hat{\mathcal{A}}^{RV}_{\geq m+1}$ we conclude that the composite $W\circ\check{Z}_{\leq m}$ is a Vassiliev invariant. Here $\check{Z}_{\leq m}$ was defined in the proof of Proposition \ref{prop:universal}. The converse to this remark also holds:

\begin{proposition}
Let $\theta$ be a degree $m$ Vassiliev invariant of rotational virtual knots. Then there exists some degree $m$ weight system $W_\theta\in \mathcal{W}_m$ such that $\theta= W_\theta\circ \check{Z}_{\leq m}$.
\end{proposition}

\begin{proof}
This proof is analogous to that for classical framed knots; see \cite{jackson2019introduction}. Noting that 
\[
    \mathcal{V}_m\cong \frac{\mathcal{V}_m}{\mathcal{V}_{m-1}} \oplus \dots \oplus \frac{\mathcal{V}_1}{\mathcal{V}_0} \oplus \mathcal{V}_0
\]
we can write
\[
    \theta = (\theta_m,\dots,\theta_1,\theta_0).
\]
Then in light of the identification $\mathcal{V}_i/\mathcal{V}_{i-1}\cong (\mathcal{K}_i/\mathcal{K}_{i+1})^*$ we can define the degree $i$ weight systems $W_i=\theta_i\circ \psi_i:\mathcal{A}^{RV}_i\to \mathbb{C}$ and define the composite weight system $W = W_m+\dots + W_0$. We claim $W$ is as required. Indeed, let $\check{Z}_i=p_i\circ \check{Z}$ where $p_i:\hat{\mathcal{A}}^{RV}\to \hat{\mathcal{A}}^{RV}_i$ is the obvious projection map. Using $\psi_i^{-1}=\check{Z}_i$ from Proposition \ref{prop:universal} we see
\[
    W_i \circ \check{Z}_i = \theta_i\circ \psi_i \circ \check{Z}_i =\theta_i
\]
for all $i$. We conclude $W_\theta\circ \check{Z}_{\leq m}=\theta$ as maps on $\mathcal{K}/\mathcal{K}_{m+1}$. Since both maps are degree $m$ Vassiliev invariants they both vanish on $\mathcal{K}_{m+1}$, so this is sufficient to conclude they are also equal as maps on $\mathcal{K}$, as required.
\end{proof}

\subsection{Lie Algebra Weight Systems}\label{subsec:lie_ws}

In this subsection we briefly discuss examples of weight systems of rotational virtual chord diagrams coming from finite-dimensional Lie algebra representations, in analogy with the theory for classical knots. Throughout this subsection all chord diagrams are assumed to be \textit{oriented}. We let $\mathfrak{g}$ denote a semisimple Lie algebra over an algebraically closed field $k$ of characteristic zero. In particular as we have already restricted our attention to $\mathbb{C}$-valued weight systems, we may take $k=\mathbb{C}$. We assume some familiarity with semisimple Lie algebras, a brief exposition of which can be found in Appendix \ref{app:lie}. 

We will follow the construction of weight systems from Lie algebra representations via a Reshetikhin-Turaev construction, using a Morse decomposition of chord diagrams. This is the approach seen for classical knots in \cite{ohtsuki2002quantum}, for example.

\begin{definition}\label{def:RT_laws}
Let $(V,\rho)$ be a finite-dimensional representation of a semisimple Lie algebra $\mathfrak{g}$. Given a rotational virtual chord diagram $C$, decompose it as a series of horizontal and vertical juxtapositions of the diagram pieces listed in Figure \ref{fig:chord_pieces}. Now, in this decomposition of $C$ we locally interpret skeleton pieces as copies of $V$ if they are directed downwards, or as copies of $V^*$ if they are directed upwards. Pieces of chord are interpreted as copies of $\mathfrak{g}$, and horizontally juxtaposed line pieces are interpreted as the tensor product of their associated copies of $V$, $V^*$, or $\mathfrak{g}$. Correspondingly we interpret an empty part of the diagram as holding a copy of $\mathbb{C}$. This is well-defined since tensor products with $\mathbb{C}$ have no effect on $V$, $V^*$, or $\mathfrak{g}$. Under this interpretation the diagram pieces in Figure \ref{fig:chord_pieces} translate to maps between these tensor products of $V$, $V^*$, and $\mathfrak{g}$, as is indicated in the figure.

\begin{figure}[ht]
    \centering
    \includegraphics[width=\linewidth]{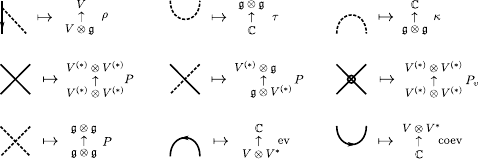}
    \caption{Elementary pieces of rotational virtual chord diagram and their representations as linear maps between copies of $\mathfrak{g}$, $V$, and $V^*$.}
    \label{fig:chord_pieces}
\end{figure}

The maps in Figure \ref{fig:chord_pieces} are defined as follows: $\kappa:\mathfrak{g}\otimes \mathfrak{g}\to \mathbb{C}$ is the Killing form, and $\tau\in \mathfrak{g}\otimes \mathfrak{g}$ is the invariant 2-tensor of $\mathfrak{g}$ (see Appendix \ref{app:lie}). Next $\text{ev}$ is the map $V\otimes V^*\to \mathbb{C}$ sending $v\otimes f$ to $f(v)\in \mathbb{C}$ and $\text{coev}:\mathbb{C}\to V\otimes V^*$ is the linear map sending $1$ to $\sum_{i=1}^{\text{dim}(V)} e_i\otimes e^i$ where $\{e_i\}$ is a basis for $V$ and $\{e^i\}$ is its dual basis for $V^*$. 
Finally $P$ is the trivial permutation $v\otimes w\mapsto w\otimes v$, and for the moment we also take the `virtual permutation' $P_v$ to be equal to $P$.

Now to construct $W_{\mathfrak{g},V}(C)$ we compose all the maps associated to our decomposition of $C$, in the order specified by this decomposition. An example of this is depicted in Figure \ref{fig:chord_RT_ex}. As rotational virtual chord diagrams are closed, the result will be a linear map $\mathbb{C}\to \mathbb{C}$. This map is uniquely determined by the image of $1$, and this image is defined to be $W_{\mathfrak{g},V}(C)$.

\begin{figure}[ht]
    \centering
    \includegraphics[width=.65\linewidth]{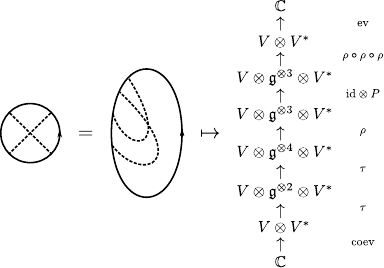}
    \caption{An example computation of the Reshetikhin-Turaev construction of $W_{\mathfrak{g},V}(C)$.}
    \label{fig:chord_RT_ex}
\end{figure}
\end{definition}

Note that by choosing $P_v=P$, we are essentially just replacing all virtual crossings by flat crossings. In this case we therefore retrieve the definition of Lie algebra weight systems of classical knots applied to virtual knots by neglecting all virtual structure. In other words, assigning the trivial permutations to virtual crossings gives rise to the following commutative diagram:
\begin{equation}\label{diag:LAWS}
\begin{tikzcd}
\mathcal{A}^{RV} \arrow[r, "F_v"] \arrow[rd, "\text{RT}"', dashed] & \mathcal{A} \arrow[d, "{W^K_{\mathfrak{g},V}}"] \\
                                                                  & \mathbb{C}                                     
\end{tikzcd}
\end{equation}
where $W^K_{\mathfrak{g},V}$ is the Lie algebra weight system for classical knots associated to $(\mathfrak{g},V)$ and $\text{RT}$ is the Reshetikhin-Turaev construction of $W_{\mathfrak{g},V}$ from Definition \ref{def:RT_laws}. 

In general, different choices of $P_v$ may be used to obtain more general Lie algebra weight systems of rotational virtual knots, that may have more distinguishing power than those we consider here. For such a choice of $P_v$ one only needs to verify that the result of Definition \ref{def:RT_laws} is invariant under the flat virtual Reidemeister moves and satisfies $CD$.

\begin{example}
Let $n$ be the dimension of the $\mathfrak{g}$-representation $V$. To satisfy the $CD$ relation it suffices to let
\[
    P_v = P\circ (M\otimes M)
\]
where $M$ is an intertwiner in $\text{End}(V)$ with respect to the action of $\mathfrak{g}$ on $V$, that is to say $M(x\triangleright v)= x\triangleright M(v)$ for all $x\in \mathfrak{g},v\in V$. One can check that this assignment also satisfies $vR3$ and $mR3$. To satisfy $vR2$ we must further have that $M^2=I$. So we obtain virtual Lie algebra weight systems for every self-inverse element of $\text{End}(V)$ intertwining the action of $\mathfrak{g}$. 

For example we may take $M$ to be diagonal with all nonzero entries equal to $\pm1$. The associated weight systems can distinguish chord diagrams on the virtual figure eight from their counterparts on an unknotted circle.
\end{example}

\section{Quantum Invariants}\label{sec:quantum}

In this section we recall the construction of universal quantum invariants of rotational virtual knots developed in \cite{kauffman2015rotational}, define Reshetikhin-Turaev invariant as evaluations of these, and give generalizations in the virtual setting.

We first recall the definition of the universal quantum invariant associated to a ribbon Hopf algebra introduced for rotational virtual links in \cite{kauffman2015rotational}. We follow the notational approach taken in \cite{bar2021perturbed,moltmaker2022new} by working with the `rotational virtual tangle category' $\texttt{RVT}$, into which rotational virtual knots naturally embed and to which quantum invariants naturally apply. 
In particular we will introduce quantum invariants associated to a ribbon Hopf algebra $A$ as factoring through a functor $Z:\texttt{RVT}\to \mathcal{H}$ from which extended quantum invariants also arise.
Here $\mathcal{H}$ is a category of formal ribbon Hopf algebra elements attached to knot diagrams. We assume familiarity with the basic theory of ribbon Hopf algebras; a brief summary of this topic can be found in Appendix \ref{app:hopf}. 

\begin{definition}
The \textbf{rotational virtual tangle category} $\texttt{RVT}$ is the category with objects the sets $\{(n)\}_{n\in\mathbb{N}}$ of $n$ unordered points, and with sets of morphisms consisting of \textit{oriented} rotational virtual tangles between these points (seen as equivalence classes of diagrams) whose tangent vectors are pointed downwards at the terminal points and whose terminal points may lie anywhere in the plane (except on other arcs of the diagram). See Figure \ref{fig:RVT_morphism} for an example.

\begin{figure}[ht]
    \centering
    \includegraphics[width=.22\linewidth]{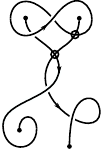}
    \caption{A morphism in $\texttt{RVT}$.}
    \label{fig:RVT_morphism}
\end{figure}
\end{definition}

The morphisms in $\texttt{RVT}$ are generated by the elementary diagram pieces depicted in Figure \ref{fig:RVT_pieces} in the following sense: any rotational virtual tangle diagram constituting a morphism in $\texttt{RVT}$ can be formed by `multiplications' of the endpoints of copies of these pieces, where a \textbf{multiplication of endpoints} is to be defined a \textbf{regular virtual arc} between them. A virtual arc is defined to be an arc that only makes virtual crossings, and can therefore be drawn in the plane in any way one pleases by virtue of the detour move. (In this category we extend the detour move to be allowed to detour across tangle end-points.) Such an arc is regular if its tangent vector makes no full turns when running along it.

\begin{figure}[ht]
    \centering
    \includegraphics[width=.5\linewidth]{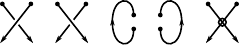}
    \caption{Elementary morphisms in $\texttt{RVT}$, denoted $R$, $\overline{R}$, $C$, $\overline{C}$, and $R_v$ in order.}
    \label{fig:RVT_pieces}
\end{figure}

To see that any knot or link can be written in terms of multiplications of endpoints on copies of the elementary morphisms in $\texttt{RVT}$, it suffices to note any rotational virtual knot can be written as a set of crossings connected by virtual arcs, and that any virtual arc can be written as copies of $C$ and $\overline{C}$ connected by regular virtual arcs.

We now introduce the functor $Z:\texttt{RVT}\to\mathcal{H}$, beginning by introducing $\mathcal{H}$.

\begin{definition}\label{def:category_h}
The category $\mathcal{H}$ is defined to be the category $\texttt{RVT}$ whose morphisms have been decorated with formal elements coming from a ribbon Hopf algebra; see Appendix \ref{app:hopf} for details on these. We also allow for the linear combination of morphisms in $\mathcal{H}$, so that each set $\text{Hom}_\mathcal{H}(x,y)$ forms a vector space over $\mathbb{C}$. The decorations on morphisms, i.e.~on tangle diagrams, can be moved around along the diagram at will, also across crossings, but not past each-other.

The elements on decorations can be multiplied according to the orientation of the tangle component they are on; see Figure \ref{fig:deco_mult}. This multiplication is understood to be the formal multiplication operation within a ribbon Hopf algebra. With respect to this multiplication the elements on decorations are assumed to satisfy the axioms of a ribbon Hopf algebra. So for example decorations with the universal $R$-matrix $\mathcal{R}$ will satisfy the quantum Yang-Baxter equation; see Lemma \ref{lm:YBE}. See \cite{kauffman2015rotational} for more details and examples of this formal calculus.



Since we can move and then multiply all the decorations on a given component together, the morphisms in $\mathcal{H}$ are equivalently rotational virtual tangles each of whose components has been given a single decoration. On closed components, this description only gives well-defined labels up to cyclic permutation of the multiplications in this decoration, and so for them we consider these labels only up to such permutations.

\begin{figure}[ht]
    \centering
    \includegraphics[width=.22\linewidth]{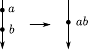}
    \caption{Multiplication of decorations on morphisms in $\mathcal{H}$.}
    \label{fig:deco_mult}
\end{figure}
\end{definition}

\begin{remark}
The formal calculus of decorations on $\mathcal{H}$ described here is somewhat different to that described in \cite{kauffman2015rotational}; particularly in how it treats curls, i.e.~cups and caps in the diagram. In \cite{kauffman2015rotational}, no algebra such as $uv^{-1}$ or $vu^{-1}$ was attached to cups or caps. Instead the rule was imposed that sliding a decoration past a cup or cap induces an application of the antipode $S$. The equivalence between these two formulations is explained in \cite[Sec.~6.6]{kauffman2015rotational} using the interplay between $S$ and $u,v$.
\end{remark}

Given our description of $\mathcal{H}$ we can say that the functor $Z$ essentially just places Hopf algebra decorations on the morphisms of $\texttt{RVT}$. The universal quantum invariant associated to a ribbon Hopf algebra $A$ will then simply consist of evaluating these expressions in $A$.

\begin{definition}\label{def:UQI}
We define the functor $Z:\texttt{RVT}\to \mathcal{H}$ by $(n)\mapsto (n)$ for objects. It is defined for morphisms to be given by Figure \ref{fig:Z_pieces} on the elementary morphisms from Figure \ref{fig:RVT_pieces}, and this assignment is extended to multiplications of these pieces in the obvious manner (namely without placing any decorations on regular virtuals arcs).

In Figure \ref{fig:Z_pieces} $\mathcal{R},\mathcal{R}^{-1},u,v$ are the defining structure morphisms of a ribbon Hopf algebra, with $\mathcal{R}$ providing a quasitriangular structure and $v$ a ribbon structure. In particular this means $\mathcal{R}$ formally represents the solution to the algebraic Yang-Baxter equation in a Hopf algebra, i.e.~it is its `universal $R$-matrix', and $\mathcal{R}^{-1}$ is its inverse. We use the notation
\[
    \mathcal{R} = \sum_i \alpha_i\otimes \beta_i
    \qquad\text{ and }\qquad
    \mathcal{R}^{-1} = \sum_i \alpha_i'\otimes \beta_i'.
\]

\begin{figure}[ht]
    \centering
    \includegraphics[width=\linewidth]{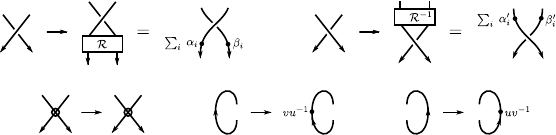}
    \caption{Definition of $Z$ on generating morphisms.}
    \label{fig:Z_pieces}
\end{figure}
\end{definition}

\begin{remark}\label{rk:evaluations}
Definition \ref{def:UQI} describes the most universal quantum invariant $Z(K)$ of a knot $K$: it keeps all the algebra formal, without specifying a ribbon Hopf algebra $A$, and also keeps the underlying knot diagram around. This is a new functor and as it is stands is not calculable, but rather serves as a point from which further evaluations can be made to describe useful quantum invariants. We have several choices of evaluation:
\begin{itemize}
    \item We can discard the topology of the underlying diagram completely, regarding it as a circle with algebra on it. The circle then imposes cyclic commutativity of the multiplication of decorations, as was noted in Definition \ref{def:category_h}, which is also known as imposing a formal trace \cite{kauffman2015rotational}. After evaluation in a Hopf algebra $A$ this option yields the standard universal quantum invariants from \cite{kauffman2015rotational}; see Definition \ref{def:quantum_invariants} below.
    \item Instead of discarding the diagram completely, we can choose to flatten its crossings to obtain a flat rotational virtual diagram. This evaluation, along with a slight generalization, constitutes the extended quantum invariants discussed below.
\end{itemize}
The verification that this functor $Z$ preserves the Reidemeister moves is done by direct generalization of known arguments \cite{Radford2,Radford1}.
\end{remark}

\begin{definition}\label{def:quantum_invariants}
\cite{kauffman2015rotational} Given a ribbon Hopf algebra $A$, we define the \textbf{universal quantum invariant} $Z_A$ associated to $A$ to be the invariant of rotational virtual knots given by the following composition:
\[
\begin{tikzcd}
\{\text{RVK}\} \arrow[r, hook] & {\text{Hom}_{\texttt{RVT}}(\emptyset,\emptyset)} \arrow[r, "Z"] & {\text{Hom}_{\mathcal{H}}(\emptyset,\emptyset)} \arrow[r] & A/I
\end{tikzcd}
\]
where the final arrow is given for a rotational virtual knot $K$ by evaluation in $A$ of the fully multiplied decoration on $Z(K)$, and $\{\text{RVK}\}$ denotes the set of rotational virtual knots. So in the words of Remark \ref{rk:evaluations}, this last map discards the underlying knot diagram in favor of a formal trace.

We must take care to note that the decoration on $Z(K)$ is only defined up to cyclic permutation of multiplication, and so for the codomain of this map we must form a corresponding quotient of $A$. Namely, we work in the quotient $A/I$ by the \textit{vector subspace} (note:~not algebra ideal) $I$ spanned by elements $\{xy-yx\in A\,\,\vert\,\,x,y\in A\}$, which imposes cyclic commutativity of the multiplication on $A/I$. Forming this quotient is otherwise known as taking a formal trace in $A$.
\end{definition}

We can alternatively move this technique through a representation $(V,\rho:A\to\text{End}(V))$ of a Hopf algebra by forming the following composition:
\[
\begin{tikzcd}
\{\text{RVK}\} \arrow[r, hook] & {\text{Hom}_{\texttt{RVT}}(\emptyset,\emptyset)} \arrow[r, "Z"] & {\text{Hom}_{\mathcal{H}}(\emptyset,\emptyset)} \arrow[r, "\rho"] & \text{End}(V) \arrow[r, "\text{Tr}"] & \mathbb{C}
\end{tikzcd}
\]
Here $\rho$ forms the representation of the decoration of $Z(K)$. The entire composition is well-defined by cyclic commutativity of traces. This therefore defines the \textbf{quantum invariant} associated to $(A,V)$, which we denote by $Q^{A,V}$. It is not difficult to show that these are equal to the quantum invariants obtained from the same information $(A,V)$ via a Reshetikhin-Turaev construction that assigns trivial permutations to virtual crossings; for example this can be shown following \cite{ohtsuki2002quantum}.

\begin{example}
A class of invariants $Q^{A,V}$ that was of interest to us in Example \ref{ex:quantum} is that arising from \textbf{quantum groups} $U_q(\mathfrak{g})$, which are ribbon Hopf algebras obtained as the $q$-deformation of the universal enveloping algebra of a semisimple Lie algebra $\mathfrak{g}$. See \cite{kassel} for more on quantum groups and Hopf algebras in general.
\end{example}

\begin{definition}\label{def:EQI}
Given a ribbon Hopf algebra $A$ we define the \textbf{extended quantum invariant} associated to $A$ to be given by the composition
\[
\begin{tikzcd}
\{\text{RVK}\} \arrow[r, hook] & {\text{Hom}_{\texttt{RVT}}(\emptyset,\emptyset)} \arrow[r, "Z"] & {\text{Hom}_{\mathcal{H}}(\emptyset,\emptyset)} \arrow[r] & A/I \times \{\text{FRV}\}
\end{tikzcd}
\]
Here $\{\text{FRV}\}$ is the set of flat rotational virtual knots. For a knot $K$, the final arrow first flattens all classical crossings in $Z(K)$ according to Remark \ref{rk:evaluations} to produce a decorated flat rotational virtual diagram. It then evaluates the decorations on $Z(K)$ as in Definition \ref{def:quantum_invariants} to produce an element of $A/I$, and also records the obtained flat virtual diagram in $\{\text{FRV}\}$.
\end{definition}

\begin{example}

In \cite{kauffman2015rotational} it was shown that there exist rotational virtual links that cannot be distinguished by any of the quantum invariants defined above. The examples of such links $L,L'$ from \cite{kauffman2015rotational} are depicted in Figure \ref{fig:problematic_knots}.

\begin{figure}[ht]
    \centering
    \includegraphics[width=.42\linewidth]{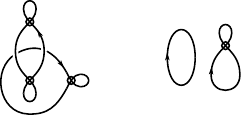}
    \caption{The problematic rotational virtual links $L$ (left) and $L'$ (right) from \cite{kauffman2015rotational}.}
    \label{fig:problematic_knots}
\end{figure}

We claim that the extended quantum invariants can distinguish the links $L$ and $L'$. Indeed, while the factor in $A/I$ of $Z'(L),Z'(L)\in (A/I) \times \{\text{FRV}\}$ cannot distinguish $L$ from $L'$ for any $A$, it is known \cite{kauffman2015rotational} that the underlying flat rotational virtuals in $\{\text{FRV}\}$ do distinguish them. Namely in \cite{kauffman2015rotational} a parity bracket invariant of flat and standard rotational virtuals was constructed using concepts discovered by Manturov \cite{Manturov}, which can distinguish these flat rotational virtuals.
\end{example}

We can make a further generalization of Definition \ref{def:UQI} by augmenting the category $\mathcal{H}$ slightly, restricting the movement of certain decorations in the image of $Z$ in a way inspired by the formalism of rotational virtual chord diagrams:

\begin{definition}
We define the augmented quantum invariant functor $Z':\texttt{RVT}\to \mathcal{H}'$ to be given by the same construction as $Z$ in Definition \ref{def:UQI}, but further placing \textit{chords} between the decorations coming from classical crossings; see Figure \ref{fig:quantum_chord}. Here $\mathcal{H}'$ is the category $\mathcal{H}$ with an augmented equivalence relation on its morphisms that we specify below. The chords have no bearing on the decorations themselves, and are purely there to keep track of which $\alpha$ belongs to which $\beta$ in the image of $Z'$. When decorations with a chord are multiplied together, we keep track of which formal factor in the multiplied expression the chord is attached to. If any decorations from classical crossings are cancelled, as happens under the $R2$ move for example, the associated chords are also discarded. 

\begin{figure}[ht]
    \centering
    \includegraphics[width=.7\linewidth]{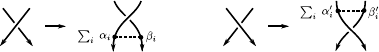}
    \caption{Placing a chord in $\mathcal{H}^{RV}$.}
    \label{fig:quantum_chord}
\end{figure}

We now define the equivalence on $\mathcal{H}'$ to be such that decorations associated to a classical crossings can only be moved across virtual arcs in pairs, much like the chords on a rotational virtual chord diagram: In $\mathcal{H}'$ we allow all the same moves on decorations as in $\mathcal{H}$, except for virtual chord slides of decorations at the end of a chord; recall the right-hand side of Figure \ref{fig:chord_sliding}. Instead, for such decorations we impose the chord detour move $CD$; recall Figure \ref{fig:CD}.
\end{definition}

The idea behind the definition of $\mathcal{H}'$ is that for invariance of $Z'$, we don't need to be able to carry individual decorations across virtual arcs; only entire crossings. This is made explicit by the following lemma:

\begin{lemma}
The augmented functor $Z'$ is an invariant of rotational virtual knots.
\end{lemma}
\begin{proof}
Since invariance of $Z$ is known, it suffices to check that the disallowing of virtual chord slides is not an obstruction to the proof of invariance of $Z$. The only (virtual) Reidemeister move where such an obstruction could occur is $mR3$, and invariance under $mR3$ is guaranteed by the $CD$ move. This is shown for a positive crossing in Figure \ref{fig:EQI_invariance}, the case for a negative crossing being analogous.
\begin{figure}[ht]
    \centering
    \includegraphics[width=.9\linewidth]{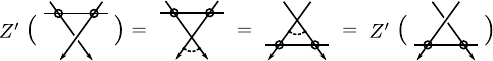}
    \caption{Proof of invariance of $Z'$ under $mR3$; decoration labels omitted for simplicity.}
    \label{fig:EQI_invariance}
\end{figure}
\end{proof}

As in Remark \ref{rk:evaluations}, while $Z'$ itself is incalculable it can be evaluated in different ways to give more tractable extended quantum invariants, for example by flattening all classical crossings. Finding different evaluations of $Z'$ (as well as $Z$) to construct new extended quantum invariants therefore makes for an interesting topic for further research.

\begin{remark}
All of the invariants discussed in this section can be generalized slightly to give (extended) quantum invariants of rotational virtual diagrams containing flat crossings as well as classical and virtual ones. Namely this is done simply by extending Figure \ref{fig:Z_pieces} to assign flat crossings to flat crossings.
\end{remark}

\section*{Acknowledgements}

We thank Dror Bar-Natan for helpful remarks.




\appendix

\section{Semisimple Lie Algebras}\label{app:lie}

We treat, without proof, the basic results on semisimple Lie algebras and their representations used in this paper. We work over a fixed field $k$, assumed to be algebraically closed and of characteristic zero for convenience.

\begin{definition}
Let $\mathfrak{g}$ be a Lie algebra. The \textbf{radical} $\mathfrak{r}$ of $\mathfrak{g}$ is defined to be the largest solvable ideal of $\mathfrak{g}$. We say $\mathfrak{g}$ is \textbf{semisimple} if $\mathfrak{r}=0$.
\end{definition}

Fundamental to the theory of semisimple Lie algebras is the Killing form:

\begin{definition}
The \textbf{Killing form} of a Lie algebra $\mathfrak{g}$ is a symmetric bilinear form $\kappa:\mathfrak{g}\otimes \mathfrak{g}\to k$ defined for $x,y\in\mathfrak{g}$ by $\kappa(x,y)=\tr\left( \text{ad}_x \circ \text{ad}_y \right)$. Here $\text{ad}_x:\mathfrak{g}\to \mathfrak{g}$ is the adjoint representation of $x$, defined by $\text{ad}_x(y)=[x,y]$.
\end{definition}

\begin{proposition}
A Lie algebra $\mathfrak{g}$ is semisimple if and only if $\kappa$ is nondegenerate.
\end{proposition}

It follows immediately that every finite-dimensional semisimple Lie algebra has a basis $\{I_\mu\}_{\mu\in\{0,\dots,m\}}$ that is orthonormal with respect to the Killing form, meaning $\kappa(I_\mu,I_\nu)= \delta_{\mu\nu}$. This basis induces a dual basis of $\mathfrak{g}^*$, denoted $\{I^\mu\}$.

\begin{definition}
Let $\mathfrak{g}$ be semisimple with $\dim(\mathfrak{g})=m$, and let $\{I_\mu\}$ be a basis for $\mathfrak{g}$ orthonormal with respect to $\kappa$. The \textbf{invariant 2-tensor} of $\mathfrak{g}$ is defined to be the element $\tau\in \mathfrak{g}\otimes \mathfrak{g}$ given by $\tau=\sum_{i=1}^m I_\mu\otimes I_\mu$. The \textbf{Casimir operator} $C$ is defined to be the image of $\tau$ under the canonical map $\mathfrak{g}\otimes \mathfrak{g}\to U(\mathfrak{g})$.
\end{definition}

\begin{lemma}
Both $\tau$ and $C$ are independent of the choice of orthonormal basis $\{I_\mu\}$.
\end{lemma}

Note that via the tensor-hom adjunction we have a series of isomorphisms
\[
    \text{Hom}(\mathfrak{g}\otimes \mathfrak{g},k)\cong \text{Hom}(k,(\mathfrak{g}\otimes \mathfrak{g})^*)\cong \text{Hom}(k,\mathfrak{g}^*\otimes \mathfrak{g}^*).
\]
The image of $\kappa \in \text{Hom}(\mathfrak{g}\otimes \mathfrak{g},k)$ under these isomorphisms is equivalent to an element $\kappa^*\in \mathfrak{g}^*\otimes \mathfrak{g}^*$.

\begin{lemma}
The invariant 2-tensor $\tau$ is the dual of $\kappa^*$ (after making the identification $(\mathfrak{g}^*\otimes \mathfrak{g}^*)^* \cong \mathfrak{g}\otimes \mathfrak{g}$).
\end{lemma}

\begin{definition}
A \textbf{representation} of a Lie algebra $\mathfrak{g}$ is a pair $(V,\rho)$ of a vector space $V$ and a Lie algebra homomorphism $\rho:\mathfrak{g}\to \text{End}(V)$.
\end{definition}

\begin{definition}
Let $\mathfrak{g}$ be a Lie algebra. The \textbf{tensor algebra} of $\mathfrak{g}$ is defined to be the space
\[
    T(\mathfrak{g}) = \bigoplus_{n=0}^\infty \mathfrak{g}^{\otimes n} = k\oplus \mathfrak{g}\oplus (\mathfrak{g}\otimes \mathfrak{g})\oplus \dots
\]
This is an algebra via the multiplication $x\cdot y=x\otimes y$. Let $I$ be the algebra ideal of $\mathfrak{g}$ generated by the set $\{x\otimes y-y\otimes x-[x,y]\,\vert\,x,y,\in\mathfrak{g}\}$. The \textbf{universal enveloping algebra} $U(\mathfrak{g})$ of $\mathfrak{g}$ is defined to be the quotient algebra $T(\mathfrak{g})/I$.
\end{definition}

Note there is a canonical map $\iota:\mathfrak{g} \to T(\mathfrak{g})\to U(\mathfrak{g})$, namely the canonical injection of $\mathfrak{g}$ into $T(\mathfrak{g})$ followed by the projection of $T(\mathfrak{g})$ onto $T(\mathfrak{g})/I$.

\begin{lemma}
Lie algebra representations of $\mathfrak{g}$ are in one-to-one correspondence with algebra representations of $U(\mathfrak{g})$. Namely if $\rho$ is a $U(\mathfrak{g})$-representation, then $\rho\circ\iota$ is a $\mathfrak{g}$-representation, and this correspondence is bijective.
\end{lemma}

\section{Ribbon Hopf Algebras}\label{app:hopf}

Here we briefly recall the definition of a ribbon Hopf algebra, outlining how the structure defining such an algebra plays into the theory of knots. As before we work over a fixed field $k$.

\begin{definition}
A \textbf{bialgebra} $B$ is a vector space over $k$ equipped with compatible algebra and coalgebra structures. Here an algebra is given by a multiplication $m:B\otimes B\to B$ and a unit $\eta: k\to B$, a coalgebra is given by the dual structure of a comultiplication $\Delta:B\to B\otimes B$ and a counit $\epsilon:B\to k$, and these structures should be compatible in the sense that $(\Delta,\epsilon)$ are algebra morphisms. A \textbf{Hopf algebra} $H$ is a bialgebra equipped with an `antipode' map $S:H\to H$ such that the following diagram commutes:
\[
\begin{tikzcd}
                                                                               & H\otimes H \arrow[r, "S\otimes \text{id}"]  & H\otimes H \arrow[rd, "m"]  &   \\
H \arrow[ru, "\Delta"] \arrow[rd, "\Delta"'] \arrow[rrr, "\eta\circ \epsilon"] &                                             &                             & H \\
                                                                               & H\otimes H \arrow[r, "\text{id}\otimes S"'] & H\otimes H \arrow[ru, "m"'] &  
\end{tikzcd}
\]
\end{definition}

Hopf algebras have a particularly nice representation theory mirroring that of groups, in some sense. Under this analogy the antipode $S$ plays the role of the group's inverse: the antipode is required to ensure that for $V$ a Hopf algebra representation, its dual $V^*$ also has a canonical Hopf algebra representation structure. Under additional assumptions this representation theory allows for the extraction of framed knot invariants. These assumptions are those of `quasi-triangularity' and `ribbon-ness':

\begin{definition}\label{def:quasitriangular}
A \textbf{quasitriangular structure} on a Hopf algebra $H$ is an element $\mathcal{R}\in H\otimes H$.
If we write $\mathcal{R}=\sum_i \alpha_i\otimes \beta_i$ and denote $\mathcal{R}_{12}=\mathcal{R}\otimes 1$, $\mathcal{R}_{23}=1\otimes \mathcal{R}$, and $\mathcal{R}_{13}=\sum_i \alpha_i\otimes 1\otimes \beta_i$, then the element $\mathcal{R}$ must satisfy:
\begin{align*}
    & P\circ \Delta(x) = \mathcal{R}\Delta(x)\mathcal{R}^{-1} \qquad \forall\,\,x\in A,\\
    & (\Delta\otimes \text{id})(\mathcal{R}) = \mathcal{R}_{13}\mathcal{R}_{23},\\
    & (\text{id}\otimes \Delta)(\mathcal{R}) = \mathcal{R}_{13}\mathcal{R}_{12},
\end{align*}
where $P: A\otimes A\to A\otimes A$ is the linear map given by $x\otimes y\mapsto y\otimes x$ on pure tensors.
\end{definition}

A Hopf algebra $H$ is called quasitriangular if it has been equipped with a quasitriangular structure $\mathcal{R}$. For a quasitriangular Hopf algebra $(H,\mathcal{R})$ we will let $u$ denote the element of $H$ given by
\[
    u = \sum_i S(\beta_i)\alpha_i.
\]

\begin{definition}
A \textbf{ribbon Hopf algebra} $A$ is a quasitriangular Hopf algebra with a distinguished central element $v\in A$ such that
\begin{align*}
    & v^2 = S(u)\cdot u,\\
    & \Delta(v) = (v\otimes v)\cdot (\mathcal{R}_{21}\mathcal{R})^{-1},\\
    & S(v)=v,\\
    & \epsilon(v)=1.
\end{align*}
\end{definition}

The definition of a quasitriangular structure is precisely such that if we assign copies of $\mathcal{R}$ or $\mathcal{R}^{-1}$ to the crossings of a knot diagram in a certain way (see Section \ref{sec:quantum} or \cite{ohtsuki2002quantum} for details), then this assignment is invariant under $R3$. This invariance follows from the following lemma, which is immediate from the definition:

\begin{lemma}\label{lm:YBE}
If $\mathcal{R}$ is a quasitriangular structure on a Hopf algebra $H$, then
\[
    \mathcal{R}_{12} \mathcal{R}_{13} \mathcal{R}_{23} = \mathcal{R}_{23} \mathcal{R}_{13} \mathcal{R}_{12}.
\]
\end{lemma}

As for the ribbon structure, the ribbon element $v$ is required to handle rotational structure in a knot diagram: in the construction of framed knot invariants from a ribbon Hopf algebra in Section \ref{sec:quantum} one can check that $v\in A$ is the element associated to a positive curl in a knot diagram. The ribbon element $v$ is also required to show $R1'$-invariance of these invariants.

\end{document}